\documentclass[12pt,reqno]{amsart}
\usepackage{graphicx}
\usepackage{subfigure}
\usepackage{latexsym}
\usepackage{amsmath}
\usepackage{amssymb}

\usepackage{amscd}
\usepackage{color}
 \newtheorem{theorem}{Theorem}[section]

 \newtheorem{Lem}[theorem]{Lemma}

\newcommand{\ba}{\begin{array}}
\newcommand{\ea}{\end{array}}
\newcommand{\beq}{\begin{equation}}
\newcommand{\eeq}{\end{equation}}

\newcommand\Let{\mathrel{\mathop:\!\!=}}

 \numberwithin{equation}{section}
 
\begin{document}

\title{A characterization of connected self-affine fractals arising from collinear digits}

\author{King-Shun Leung} \address{Department of Mathematics and Information Technology, The Education University of Hong Kong, Hong Kong}
\email{ksleung@eduhk.hk}

\author{Jun Jason Luo}\address{College of Mathematics and Statistics, Chongqing University, Chongqing 401331, China}\email{jun.luo@cqu.edu.cn}

%
%
%
%
\thanks{The research is supported by the grants of the Education University of Hong Kong (RG16/2014-2015R, RG94/2015-16, MIT/SRG10/15-16),  the NNSF of China (No.11301322), the Fundamental and Frontier Research Project of Chongqing (No.cstc2015jcyjA00035)}
%

\begin{abstract}
Let $A$ be an expanding integer matrix with  characteristic polynomial $f(x)=x^{2}+px+q$, and let $\mathcal{D}=\{0,1,\dots,|q|-2,|q|+m\}\mathbf{v}$ be a collinear digit set where $m\geqslant 0, {\mathbf v}\in {\mathbb Z}^2$. It is well known that there exists a unique self-affine fractal $T$ satisfying $AT=T+\mathcal{D}$. In this paper, we give a complete characterization on the connected $T$. That generalizes the previous result of $|q|=3$.
\end{abstract}

\maketitle

\section{Introduction}
Given an $n\times n$ integer matrix $A$, we assume it is expanding, i.e., its eigenvalues all have moduli strictly larger than $1$.  Let
${\mathcal{D}}=\{{\mathbf d}_1,\dots, {\mathbf d}_k\}\subset {\mathbb{R}}^n$ be a \emph{digit set}. It is well known that there exists a unique attractor $T\Let T(A,{\mathcal{D}})$ \cite{LaWa} satisfying:
\begin{equation}\label{eq-radixexp}
T= A^{-1}(T +
{\mathcal{D}})=\left\{\sum_{i=1}^{\infty}A^{-i}{\mathbf d}_{j_i}: {\mathbf d}_{j_i}\in
{\mathcal{D}}\right\}.
\end{equation}
We often call $T$ a \emph{self-affine fractal}. If moreover, $|\det(A)|=k$ and the interior of $T$ is nonempty,  then $T$ can tile the whole space ${\mathbb{R}}^n$ by translations. We call such $T$  a \emph{self-affine tile}.

The fundamental theory and applications of self-affine fractals/tiles have been extensively studied in the literature (\cite{K},\cite{LaWa},\cite{LaWa2},\cite{LaWa3},\cite{GrHa},\cite{HaSaVe},\cite{BaGe}). In the studies,  people found that, given a matrix $A$, the structures of digit set $\mathcal D$ strongly influence the topological properties of $T(A,{\mathcal{D}})$, such as connectedness and disk-likeness (see \cite{AkGj},\cite{AT},\cite{BaWa},\cite{DeLa},\cite{Ha},\cite{Ki},\cite{KiLa},\cite{KiLaRa},\cite{LeLa},\cite{LeLu},\cite{LeLu2},\cite{LLT},\cite{LLX}). Among all the researches, the collinear digit sets perhaps attracted the most attentions. Say $\mathcal{D}$ is collinear if $\mathcal{D}=\{d_1, \dots, d_k\}{\mathbf v}$ for some vector ${\mathbf v}\in {\mathbb R}^n$ and $d_1<d_2<\cdots<d_k$. If moreover,  $d_{i+1}-d_i=1$ for all $i$, ${\mathcal D}$ is said to be {\it consecutive collinear (CC)}. If $d_{i+1}-d_i=1$ for all $i$ except one $i_0$ where $d_{i_0+1}-d_{i_0}>1$, then ${\mathcal D}$ is said to have a {\it jump}. The study on the connected self-affine fractals/tiles arising from  CC digit sets has been an interesting topic.  Hacon {\it et al.} \cite{HaSaVe} first proved that a self-affine tile $T$ is always pathwise connected when $k=2$. Lau and his coworkers (\cite{HeKiLa},\cite{KiLa},\cite{KiLaRa},\cite{LeLa},\cite{LeLu3}) developed this direction and  systematically studied  the topology of self-affine tiles for any $k$. The connectedness of self-affine fractals with CC digit sets was also concerned in \cite{LeLu},\cite{LLX}.

However, there are very limited results on the collinear digit set $\mathcal D$ with jumps. In \cite{LeLu}, the authors made a first attempt in this area,  especially we proved that

\begin{theorem} \label{th-LL}
Let $A$ be an expanding integer matrix with characteristic polynomial $f(x)=x^2+px\pm 3$,  and let ${\mathcal{D}}=\{0,1,b\}{\mathbf v}$  where $2 \leqslant b\in {\mathbb{Z}}$ such that $\{{\mathbf v}, A{\mathbf v}\}$ is linearly independent. Then we have

(i) when $b=2$, $T$ is always a connected self-affine tile;

(ii) when $b\geqslant 4$, $T$ is always a disconnected self-affine fractal;

(iii) when $b=3$, $T$ is  connected if and only if $(p,q)\in\{(\pm1,-3), (\pm2,3),(\pm3,3)\}$.
\end{theorem}

For an expanding $2\times 2$ integer matrix $A$, it is known by \cite{BaGe} that the characteristic polynomial of $A$ is
given by
\begin{align} \label{eq.expanding}
f(x)=x^2+px+q, ~\text{with}~ |p|\leqslant q, ~\text{if}~ q\geqslant 2;\quad |p|\leqslant |q+2|, ~\text{if}~ q\leqslant -2.
\end{align}
In the paper,  we will give a complete characterization on the connectedness of $T$ arising from a collinear digit set with a jump. As for $|q|=2$ and ${\mathcal D}=\{{\mathbf 0}, {\mathbf v}\}$, $T$ is always a connected self-affine tile (\cite{HaSaVe} or \cite{KiLa}). So we will exclude this trivial case.

\begin{theorem}\label{mainthm}
Let $A$ be an expanding integer matrix with  characteristic polynomial $f(x)=x^2+px+q$ where $|q|\geqslant 3$, let $\mathcal{D}=\{0,1,\dots,|q|-2,|q|+m\}\mathbf{v}$, where $0\leqslant m\in{\mathbb Z}$ and $\mathbf{v}\in\mathbb{Z}^{2}$ such that $\{{\mathbf v}, A{\mathbf v}\}$ is linearly independent. Then

(i)  when $m\geqslant 1$, $T$ is always disconnected;

(ii) when $m=0$, $T$ is connected if and only if
$$(p,q)\in  \{ (p,q)\in\mathbb{Z}^{2}:2|p|=|q+2|\} \cup\{(\pm1,-3), (\pm2,3),(\pm3,3),(\pm4,4)\}.$$ (See Figure \ref{fig1})
\end{theorem}
\begin{figure}[h]
  \centering
   \includegraphics[width=6cm]{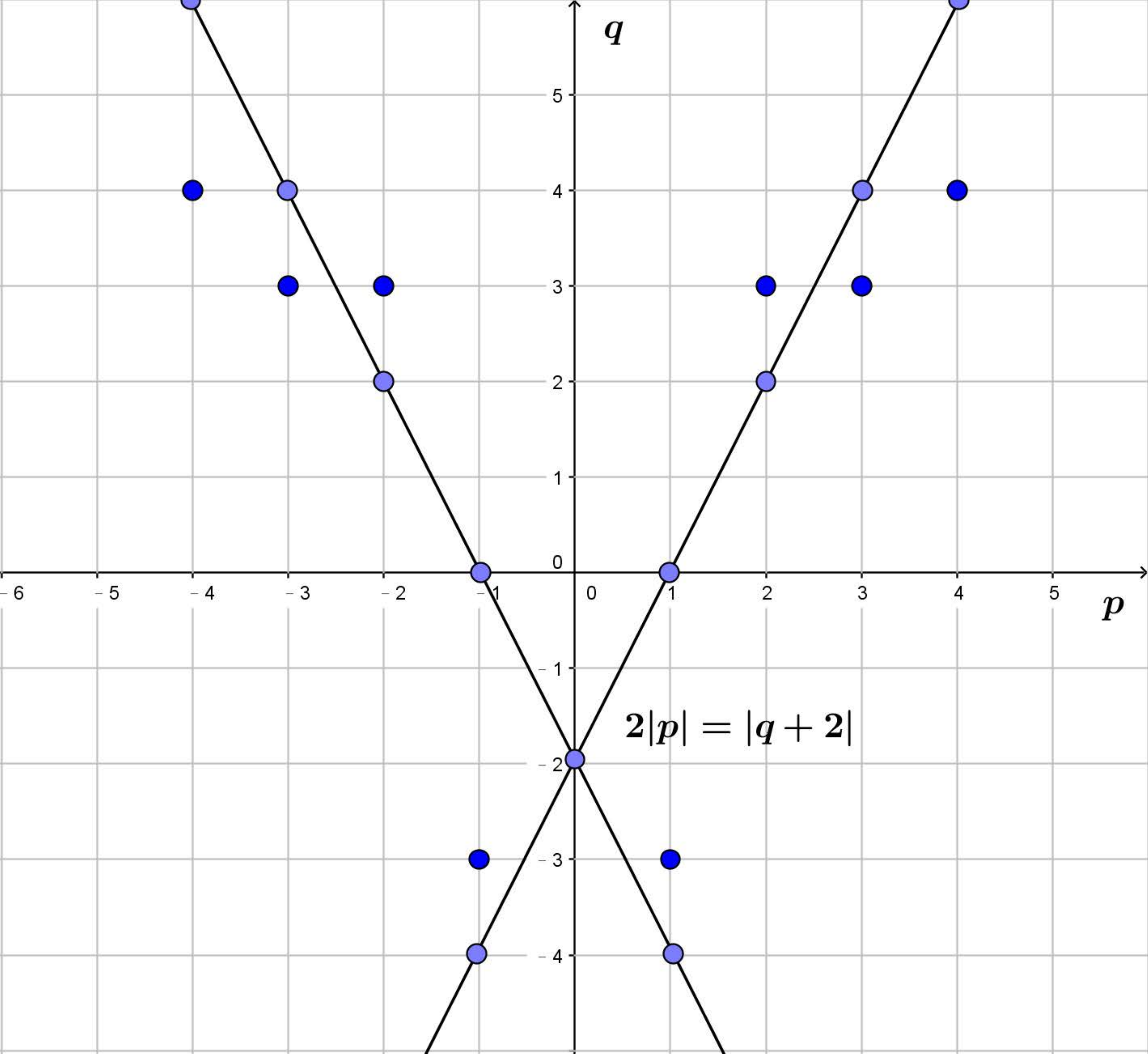}
 \caption{The domain of $(p,q)$ for connected self-affine fractals.}\label{fig1}
\end{figure}

The proof is elementary although it contains lots of calculations and multiple discussions. The main idea is to fully use the radix expansion like \eqref{eq-radixexp} and Cayley-Hamilton theorem (see basic tools in Section 2). Moreover, we remark  that when $m=-1$, ${\mathcal D}$ becomes a CC digit set and $T$ is always connected \cite{KiLa}. We also mention that the theorem still holds if the jump occurs elsewhere. Actually the proof is the same irrespective of the location of the jump occurs for the disconnected cases as it does not involve finding the exact radix expansion. For the connected cases when $m=0$, ${\mathcal D}-{\mathcal D}$ is unchanged wherever the jump occurs.

For the organization of the paper, we provide some useful lemmas in Section 2 and prove Theorem \ref{mainthm} by five parts in Section 3.

\section{Basic lemmas}

In this section, we prepare several basic results that will be used frequently in the next section. Let $(A, {\mathcal D})$ be given as in the assumption of Theorem \ref{mainthm}. Denoted by $D=\{0, 1, \dots, |q|-2, |q|+m\}$ and $\triangle D=D-D$, then ${\mathcal D}= D{\mathbf v}$ and $\triangle{\mathcal D}=\triangle D {\mathbf v}$.

First we provide a simple but useful criterion for connectedness of $T(A,{\mathcal D})$ (we refer to \cite{Ha} or \cite{KiLa} for its general version and proof).

\begin{Lem}\label{e-connected prop}
$T:=T(A,{\mathcal D})$ is connected if and only if $\mathbf{v}\in T-T$ and $(m+c)\mathbf{v}\in T-T$ for some $c\in\{2,3,\dots, |q|\}$.
\end{Lem}

Let $\Delta=p^2-4q$ be the discriminant of the polynomial $f(x)=x^2+px+q$, and define $\alpha_i,\beta_i$ by
$$
A^{-i}\mathbf{v}=\alpha_i\mathbf{v}+\beta_iA\mathbf{v}, \quad i=1,2,\dots
$$
From the Cayley-Hamilton theorem $f(A)=A^2+pA+qI=0$, where $I$ is the identity matrix, the following lemma is immediate.

\begin{Lem} \cite{Le} \label{evaluation}
Let $\alpha_i,\beta_i$ be defined as the above. Then $q\alpha_{i+2}+p\alpha_{i+1}+\alpha_i=0$ and $q\beta_{i+2}+p\beta_{i+1}+\beta_i=0$, i.e.,
$$\left[
        \begin{array}{rr}
          \alpha_{i+1} \\
          \alpha_{i+2} \\
    \end{array}
    \right]=\left[
        \begin{array}{rrrr}
          0  & 1 \\
           -1/q &  -p/q   \\
    \end{array}
    \right]^i \left[
        \begin{array}{rr}
          \alpha_1 \\
          \alpha_2 \\
    \end{array}
    \right]; \quad \left[
        \begin{array}{rr}
          \beta_{i+1} \\
          \beta_{i+2} \\
    \end{array}
    \right]=\left[
        \begin{array}{rrrr}
          0  & 1 \\
           -1/q &  -p/q   \\
    \end{array}
    \right]^i \left[
        \begin{array}{rr}
          \beta_1 \\
          \beta_2 \\
    \end{array}
    \right]
$$
and $\alpha_1=-p/q, \alpha_2=(p^2-q)/q^2; \beta_1=-1/q, \beta_2=p/q^2$. Moreover for $\Delta\ne 0$, we have $$\alpha_i=\frac{q(y_1^{i+1}-y_2^{i+1})}{\Delta^{1/2}} \quad\text{and}\quad \beta_i=\frac{-(y_1^i-y_2^i)}{\Delta^{1/2}},$$
where $y_1=\frac{-p+\Delta^{1/2}}{2q}$ and $y_2=\frac{-p-\Delta^{1/2}}{2q}$ are the two roots of $qx^2+px+1=0$.
\end{Lem}

Write
\begin{equation*}
\tilde{\alpha}\Let \sum_{i=1}^{\infty}|\alpha_i|,\quad \tilde{\beta}\Let \sum_{i=1}^{\infty}|\beta_i|.
\end{equation*}
Then the $\tilde{\alpha}, \tilde{\beta}$ are finite numbers as $|y_1|<1, |y_2|<1$. The following are their values or estimates in detail.

\begin{Lem}\cite{LLX}\label{lem-LiuLuoXie}
If $\Delta=p^2-4q\geqslant 0$. Then
\begin{equation*}
\tilde{\alpha}= \left\{
\begin{array}{ll}
\frac{|p|-1}{q-|p|+1} &\quad  q>0 \\ \\
\frac{|p|+1}{|q|-|p|-1} & \quad  q<0;
\end{array}
\right. \qquad
\tilde{\beta}= \left\{
\begin{array}{ll}
\frac{1}{q-|p|+1} &\quad q>0 \\ \\
\frac{1}{|q|-|p|-1} &  \quad  q<0.
\end{array}
\right.
\end{equation*}
\end{Lem}

On the other hand, if $\Delta< 0$, then $$|\alpha_i|\leqslant\frac{2q|y_1^{i+1}|}{|\Delta^{1/2}|}=\frac{2q^{-(i-1)/2}}{(4q-p^2)^{1/2}}\quad\text{and}\quad|\beta_i|\leqslant
\frac{2|y_1^i|}{|\Delta^{1/2}|}=\frac{2q^{-i/2}}{(4q-p^2)^{1/2}}.$$
The upper bounds of $\tilde{\alpha}, \tilde{\beta}$ are estimated by:
\begin{eqnarray*}
\tilde{\alpha}  & \leqslant & \sum_{i=1}^{n-1}|\alpha_i|+
\frac{2q^{-(n-1)/2}}{(1-q^{-1/2})(4q-p^2)^{1/2}}, \\
\tilde{\beta}  &\leqslant & \sum_{i=1}^{n-1}|\beta_i|+
\frac{2q^{-n/2}}{(1-q^{-1/2})(4q-p^2)^{1/2}}.
\end{eqnarray*}
We can find very accurate upper bounds of $\tilde{\alpha}$ and $ \tilde{\beta}$ by taking proper $n$. This is the most important tool in our proofs.

Let $L:=\{\gamma \mathbf{v}+\delta A\mathbf{v}: \gamma,\delta\in {\mathbb{Z}}\}$ be the \emph{lattice} generated by $\{\mathbf{v},A\mathbf{v}\}$. For $l\in L\setminus\{\mathbf{0}\},~ T+l$ is called a \emph{neighbor} of $T$ if $T\cap(T+l)\ne \emptyset.$ It is easy to see that $T+l$ is a neighbor of $T$ if and only if $l$ can be expressed as
$$l=\sum_{i=1}^{\infty}b_iA^{-i}\mathbf{v}\in T-T, ~\text{where}~ b_i\in \triangle D.$$
Suppose $T+l$ is a neighbor of $T$, where $l=\sum_{i=1}^{\infty}b_iA^{-i}\mathbf{v}:=\gamma \mathbf{v}+\delta A\mathbf{v}$, then
\begin{equation}\label{eq-gamma-delta}
|\gamma|\leqslant\max_i|b_i|\tilde \alpha \quad\text{and}\quad |\delta|\leqslant \max_i|b_i|\tilde \beta.
\end{equation}

Multiplying $A$ on both sides of the expression of $l$ and applying $f(A)=0$,  it follows  that $T+l_1$ is also a neighbor of $T$ satisfying $l_1=Al-b_1v=\gamma_1\mathbf{v}+\delta_1 A\mathbf{v}$ with $\gamma_1=-(q\delta+b_1)$ and $\delta_1=\gamma-p\delta$.  Continuing this process, we can construct a sequence of neighbors of $T$: $\{T+l_n\}_{n=0}^{\infty}$, where $l_0=l$ and $l_n=\gamma_n \mathbf{v}+\delta_n A\mathbf{v}, n\geq 1$ by the following  neighbor-generating  formula:
\begin{equation}\label{neighbor-generating}
\left[
        \begin{array}{rr}
          \gamma_n \\
          \delta_n \\
    \end{array}
    \right]=A^n \left[
        \begin{array}{rr}
          \gamma \\
          \delta \\
    \end{array}
    \right]-\sum_{i=1}^n A^{i-1}\left[
        \begin{array}{rr}
          b_{n+1-i} \\
          0\\
    \end{array}
    \right].
\end{equation}
Moreover, $|\gamma_n|\leqslant \max_i|b_i|\tilde \alpha$ and $|\delta_n|\leqslant \max_i|b_i|\tilde \beta$ hold for any $n\geqslant 0$. Our main idea of proof in the next section  is to find contradictions with \eqref{eq-gamma-delta} by using the neighbor-generating formula \eqref{neighbor-generating}.

We remark that by using the neighbor language, Lemma \ref{e-connected prop} can be rewritten as: \emph{$T$ is connected if and only if both  $T+\mathbf{v}$  and  $T+(m+c)\mathbf{v}$ are neighbors of  $T$ for some $c\in\{2,3,\dots, |q|\}$.}

\begin{Lem}\label{lem1}
Let $T_1=T(A,{\mathcal{D}})$ and $T_2=T(-A,{\mathcal{D}})$. Then $T_1+l$ is a neighbor of $T_1$ if and only if $T_2+l$ is a neighbor of $T_2$. Hence $T_1$ is connected if and only if $T_2$ is connected.
\end{Lem}

\begin{proof}
If $l\in T_1-T_1$, then
$$l=\sum_{i=1}^{\infty}b_iA^{-i}\mathbf{v}=\sum_{i=1}^{\infty}b_{2i}(-A)^{-2i}\mathbf{v}+\sum_{i=1}^{\infty}(-b_{2i-1})(-A)^{-2i+1}\mathbf{v}.$$
Hence $l \in T_2-T_2$  and vice versa. The second part follows from Lemma \ref{e-connected prop}.
\end{proof}

The last lemma is a special case of Theorem 1.3 in \cite{LeLu}.

\begin{Lem}\label{lem2}
Under the same assumption of Theorem \ref{mainthm}. If $p=0$, then $T$ is connected if and only if ${\mathcal D}=\{0,1,\dots, |q|-1\}{\mathbf v}$.
\end{Lem}

\section{Proof of the main theorem}

The proof of Theorem \ref{mainthm}, which consists of five parts, is given in this section. Parts I and II deal with the case when $m\geqslant 1$ while parts III to V the case $m=0$. In view of Lemmas \ref{lem1}, \ref{lem2} and Theorem \ref{th-LL},  it suffices to show the theorem under the assumption that $|q|\geqslant 4$ and $p\geqslant 1$.

\vspace{12pt}

\noindent  $\bullet$\;\textbf{Part I}

\noindent In this part we assume that $f(x)=x^{2}+px+q$ where $q\geqslant4$ and $p\geqslant1$. Let $D=\{0,1,2,\dots,q-2,q+m\}$ where $m\geqslant1$.
The proof here is then divided into two cases: $\Delta=p^{2}-4q\geqslant0$ and $\Delta<0$. Moreover, in Part I, we will  explain carefully how to use the neighbor-generating formula \eqref{neighbor-generating} to get contradictions, then we shall omit the details in the remaining parts for simplification.

\vspace{12pt}

\noindent  \underline{\textbf{Case A: $\Delta\geqslant 0$}}

\noindent  Since $A$ is expanding, $q\geqslant p$ by \eqref{eq.expanding}. By Lemma \ref{lem-LiuLuoXie} and \eqref{eq-gamma-delta}, for any $l\Let\gamma {\mathbf v}+\delta A{\mathbf v}\in T-T$, we have
\begin{equation}\label{eq.3.0}
|\gamma|\leqslant\dfrac{(q+m)(p-1)}{q-p+1}, \quad |\delta|\leqslant\dfrac{q+m}{q-p+1}.
\end{equation}
Now suppose
\begin{equation}\label{eq.3.1}
(m+c)\mathbf{v}=\sum_{i=1}^{\infty}b_{i}A^{-i}\mathbf{v}\in T-T,\mbox{\; where \;} c\geqslant2, \;   b_i\in \triangle D.
\end{equation}
By multiplying $A$, then
\[
(m+c)A\mathbf{v}-b_{1}\mathbf{v}\in T-T.
\]
If $(m+c)p-c<(m+c-1)q$, then $m+c>\dfrac{q+m}{q-p+1}$ contradicting \eqref{eq.3.0}. Hence $(m+c)\mathbf{v}\notin T-T$.  If
\begin{equation}\label{eq.3.2}
(m+c-1)q\leqslant(m+c)p-c.
\end{equation}
By multiplying $A$ again,  we have
\begin{equation*}
(m+c)A^{2}\mathbf{v}-b_{1}A\mathbf{v}-b_{2}\mathbf{v}\in T-T.
\end{equation*}
Applying $f(A)=A^2+pA+qI=0$ to reduce the left term above, it follows that
\begin{equation}\label{eq.3.3}
-[p(m+c)+b_{1}]A\mathbf{v}-[q(m+c)+b_{2}]\mathbf{v}\in T-T.
\end{equation}
By \eqref{eq.3.0}, then
\begin{equation}\label{eq.3.4}
|p(m+c)+b_{1}|\leqslant\frac{q+m}{q-p+1}.
\end{equation}
When $q=p$, \eqref{eq.3.4}  becomes $|q(m+c)+b_{1}|\leqslant q+m$. So
\begin{align*}
b_{1} \leqslant q+m-q(m+c)=-q-m(q-1)-q(c-2) <-q-m,
\end{align*}
which implies $b_{1}\notin\triangle D$ and thus $(m+c)\mathbf{v}\notin T-T$. When $q\geqslant p+1$, we can deduce from \eqref{eq.3.4}  that $|p(m+c)+b_{1}|\leqslant\dfrac{q+m}{2}$.
So
\begin{align*}
b_{1} & \leqslant\frac{q+m}{2}-p(m+c)\\
 & \leqslant\frac{q+m}{2}-q(m+c-1)-c\mbox{\quad (by \eqref{eq.3.2})}\\
 & =-q-q\left(m+c-\frac{5}{2}\right)+\frac{m}{2}-c\\
 & \leqslant-q-4\left(m+c-\frac{5}{2}\right)+\frac{m}{2}-c\\
 & =-q-\frac{7m}{2}-5(c-2)\\
 & <-q-m,
\end{align*}
implying $b_{1}\notin\triangle D$. Thus $(m+c)\mathbf{v}\notin T-T$ and $T$ is disconnected by Lemma \ref{e-connected prop}.

\vspace{12pt}

\noindent\underline{\textbf{Case B: $\Delta < 0$}}

\noindent Since the conditions $\Delta<0$ and $q=p$ imply $q<4$, we  only need to consider the case that $q\geqslant p+1$. When $q=p+1$,
there are two possibilities: $(p,q)=(3,4)$ or $(4,5)$.

\vspace{12pt}

In the case that $(p,q)=(3,4)$, we can estimate the upper bound $\tilde{\beta}<0.56$ by using the formulas in Section 2. Then $|\delta|<0.56(4+m)$. The assumption of \eqref{eq.3.1} is invalid because for any $m\geqslant1$ and any $c\geqslant2$ we have $0.56(4+m)<m+c$. Hence $(m+c)\mathbf{v}\notin T-T$.

In the case that $(p,q)=(4,5)$, we find $\tilde{\beta}<0.6$. We can deduce from \eqref{eq.3.3} that
$$-(4(m+c)+b_{1})A\mathbf{v}-(5(m+c)+b_{2})\mathbf{v}\in T-T.$$
Then $|4(m+c)+b_{1}|<0.6(5+m)$ by \eqref{eq.3.0}. So $b_1<-5-m$ as $c\geqslant2$, implying $b_1\notin \triangle D$. Hence $(m+c)\mathbf{v}\notin T-T$.

\vspace{12pt}

When $q=p+2$, the possible $(p,q)$'s are: $(2,4),(3,5),(4,6)$ and $(5,7)$. In the case that $(p,q)=(2,4)$, we calculate $\tilde{\beta}<0.5$ and so $|\delta|<0.5(4+m)$. Under the assumption of \eqref{eq.3.1}, we get $(m+c)A\mathbf{v}-b_{1}\mathbf{v}\in T-T$ as before. While $0.5(4+m)<m+c$ holds for any $m\geqslant1$ and $c\geqslant2$, contradicting \eqref{eq.3.0}. Hence $(m+c)\mathbf{v}\notin T-T$. Similarly we can do the other three cases by the upper bounds of $\tilde{\beta}$ to be $0.4, 0.4, 0.34$, respectively. We can also get $(m+c)\mathbf{v}\notin T-T$.

\vspace{12pt}

When $q\geqslant p+3$, we can obtain
\begin{equation}\label{eq3.6}
|\delta|\leqslant\frac{q+m}{q-p-1}
\end{equation}
and
\begin{equation}\label{eq3.7}
|\gamma|\leqslant\begin{cases}
\dfrac{(q+m)(p-1)}{q-p-1} & \text{if $ p^{2}\geqslant q$}\\
\dfrac{(q+m)(qp-2p^{2}+q)}{q(q-p-1)} & \text{if $ p^{2}< q$.}
\end{cases}
\end{equation}
In order to derive \eqref{eq3.6} and \eqref{eq3.7}, it suffices to show
$$\tilde{\alpha}\leqslant\dfrac{p-1}{q-p-1}\quad \text{and}\quad \tilde{\beta}\leqslant\dfrac{1}{q-p-1}.$$
Indeed, by Lemma \ref{evaluation}, we have $q\beta_{i+2} + p\beta_{i+1}+\beta_{i}=0$, and hence $|q| |\beta_{i+2}| \leq |p||\beta_{i+1}| + |\beta_{i}|$. Summing the inequality for $i=1,2,3\ldots$, we get $ |q|(\widetilde{\beta}-|\beta_{1}| - |\beta_{2}|) \leq |p|(\widetilde{\beta}-|\beta_{1}|)+\widetilde{\beta}$. Note that $|\beta_{1}| = 1/q, |\beta_{2}| = p/q^2$. It follows that
$$
(|q|-|p|-1)\widetilde{\beta} \leq (|q|-|p|)|\beta_{1}| + |q| |\beta_{2}| =1.
$$
Using a similar method, we obtain $(q-p-1)\tilde{\alpha}\leqslant(q-p)|\alpha_{1}|+q|\alpha_{2}|$. By substituting $\alpha_{1}={-p}/{q}$ and $\alpha_{2}=(p^{2}-q)/{q^{2}}$ into the inequality, we will get the desired upper bound for $\tilde{\alpha}$.

\vspace{12pt}

\noindent (a) $p^{2}\geqslant q$.

If $2p<q$, then $\dfrac{p-1}{q-p-1}<1$. It follows from \eqref{eq.3.3} and \eqref{eq3.7} that $q(m+c)+b_{2}<q+m$. So $b_{2}<-m(q-1)-q(c-1)<-m-q$, implying $b_2\notin \Delta D$. Hence $m+c\notin T-T$. If $q\leqslant 2p$. Notice first that
$0<\dfrac{p^{2}}{4}<q\leqslant2p$ and $q\geqslant p+3$, hence the possible $(p,q)$'s are
$$
(3,6),(4,7),(4,8),(5,8),(5,9),(5,10),(6,10),(6,11),(6,12),(7,13),(7,14).
$$

By calculating the upper bounds of $\tilde{\beta}$'s one by one carefully, we can take $\tilde{\beta}< 0.3$ for the first five $(p,q)$'s and $\tilde{\beta}\leqslant 0.2$ for the last six ones. Moreover, for each case $|\delta|< m+c$ always holds. Hence  $(m+c)\mathbf{v}\notin T-T$.

\vspace{12pt}

\noindent (b) $p^{2}<q$.

\vspace{12pt}

(i) $p=1$.  From \eqref{eq3.6}, we have
\[
|\delta|\leqslant\frac{q+m}{q-2}\leqslant 1+\frac{m+2}{2}<m+c.
\]
So $(m+c)\mathbf{v}\notin T-T$.

\vspace{12pt}

(ii) $p=2$.  From \eqref{eq3.7}, we have
\[
|\gamma|\leqslant\frac{(q+m)(3q-8)}{q(q-3)}=3+\frac{3m+1}{q-3}-\frac{8m}{q(q-3)}<3+3m.
\]
Hence $|q(m+c)+b_{2}|<3+3m$ and thereby $b_{2}<3-qc-m(q-3)<-q-m$, implying $b_2\notin \Delta D$. So $(m+c)\mathbf{v}\notin T-T$.

\vspace{12pt}

(iii) $p\geqslant3$.  From \eqref{eq3.7}, we have
\begin{align*}
|\gamma| & \leqslant\frac{(q+m)\left[q(p+1)-2p^{2}\right]}{q(q-p-1)} =\frac{(p+1)q^{2}+\left[(p+1)m-2p^{2}\right]q-2p^{2}m}{q(q-p-1)}\\
 & =\frac{(p+1)q^{2}-(p+1)^{2}q+\left[(p+1)m-2p^{2}+(p+1)^{2}\right]q-2p^{2}m}{q(q-p-1)}\\
 & =p+1+\frac{(p+1)m-p^{2}+2p+1}{q-p-1}-\frac{2p^{2}m}{q(q-p-1)}\\
 & <p+1+\frac{(p+1)m+2(p+1)}{q-p-1}\\
 & <p+1+m+2 \quad \mbox{(as \ensuremath{q>p^{2}>2p+2} for \ensuremath{p\geqslant3}, ie. \ensuremath{p+1<q-(p+1)})}\\
 & \leqslant q+m.
\end{align*}
Hence $|q(m+c)+b_{2}|\leqslant q+m$ and thereby $b_{2}\leqslant-m(q-1)-q(c-1)<-m-q$, implying $b_2\notin \Delta D$. So $(m+c)\mathbf{v}\notin T-T$. We prove that $T$ is disconnected by Lemma \ref{e-connected prop}.

\vspace{12pt}

\noindent $\bullet$\;\textbf{Part II}

\noindent In this part we assume that $f(x)=x^{2}+px-q$, where $q\geqslant4$. Since $A$ is expanding, $p\leqslant q-2$ by \eqref{eq.expanding}. We also see that $\Delta=p^{2}+4q\geqslant0$ always holds. Let $D=\{0,1,2,\dots,q-2,q+m\}$ where $m\geqslant1$.  Note that
\[
\tilde{\alpha}=\frac{p+1}{q-p-1},\quad \tilde{\beta}=\frac{1}{q-p-1}.
\]
Hence
\[
|\gamma|\leqslant\frac{(q+m)(p+1)}{q-p-1},\quad  |\delta|\leqslant\frac{q+m}{q-p-1}.
\]

Assume $(m+c)\mathbf{v}=\sum_{i=1}^{\infty}b_{i}A^{-i}\mathbf{v}\in T-T$ where $c\geqslant 2$. So $(m+c)A\mathbf{v}-b_{1}\mathbf{v}\in T-T$. If  $p(m+c)+(2m+c)<q(m+c-1)$ then $|\delta|\leq\dfrac{q+m}{q-p-1}<m+c$, impossible. Thus $(m+c)\mathbf{v}\notin T-T$. On the other hand, we consider the case
\begin{equation}\label{eq-ineq}
q(m+c-1)\leqslant p(m+2)+(2m+c).
\end{equation}
Analogously to Part I, by using the neighbor-generating formula \eqref{neighbor-generating}, we can deduce from $(m+c)\mathbf{v}\in T-T$ that
\begin{equation*}
-[(m+c)p+b_{1}]A\mathbf{v}+[(m+c)q-b_{2}]\mathbf{v}\in T-T.
\end{equation*}
Thus
\[
|(m+c)p+b_{1}|\leqslant\frac{q+m}{q-p-1} \quad\mbox{and}\quad |(m+c)q-b_{2}|\leqslant\frac{(q+m)(p+1)}{q-p-1}.
\]
There are two subcases to be considered here: (a) $q=p+2$ and (b) $q\geqslant p+3$.

\vspace{12pt}

\noindent (a) In this case $|\gamma|\leqslant(q+m)(p+1)=(q+m)(q-1)$ and $|\delta|\leqslant q+m$. Since $(m+c)\mathbf{v}=\sum_{i=1}^{\infty}b_{i}A^{-i}\mathbf{v}\in T-T$, it follows from \eqref{neighbor-generating} that
\begin{equation}\label{eq3.8}
[(m+c)(p^{2}+q)+pb_{1}-b_{2}]A\mathbf{v}-[qp(m+c)+qb_{1}+b_{3}]\mathbf{v}\in T-T.
\end{equation}
As $|(m+c)(p^{2}+q)+pb_{1}-b_{2}|\leqslant q+m$, so
\begin{align*}
pb_{1} & \leqslant q+m-(m+c)(p^{2}+q)+b_{2}\\
 & \leqslant q+m-(m+c)p^{2}-(m+c)q+q+m\\
 & \leqslant -mq-(m+2)p^{2}+2m.
\end{align*}
Thus for any $q\geqslant4$ we have
\begin{equation*}
b_{1}  \leqslant\frac{-mq}{p}-(m+2)p+\frac{2m}{p} =\frac{-mq}{q-2}-(m+2)(q-2)+\frac{2m}{q-2} <-q-m,
\end{equation*}
implying $b_1\notin \Delta D$. Hence $(m+c)\mathbf{v}\notin T-T$ for all integers $c\geqslant 2$.

\vspace{12pt}

\noindent (b) In this case we have
\[
|\gamma|\leqslant\frac{(q+m)(p+1)}{q-p-1}\leqslant\frac{(q+m)(p+1)}{2}\quad\mbox{and}\quad |\delta|\leqslant\frac{q+m}{q-p-1}\leqslant\frac{q+m}{2}.
\]
From \eqref{eq3.8}, we have $|(m+c)(p^{2}+q)+pb_{1}-b_{2}|\leqslant\dfrac{q+m}{2}$. So
\begin{align*}
pb_{1} &\leqslant \frac{q+m}{2}-(m+c)(p^2+q)+b_2 \\
       &\leqslant \frac{q+m}{2}-(m+2)(p^2+q)+b_2 \\
       &\leqslant -\left(m+\frac{1}{2}\right)q-(m+2)p^{2}+\frac{3m}{2}\quad \mbox{(as \ensuremath{b_{2}\leqslant q+m})}.
\end{align*}
Hence for $m\geqslant1$ and $q\geqslant4$ we have
\begin{align*}
b_{1} & \leqslant\frac{-\left(m+\frac{1}{2}\right)(p+3)}{p}-(m+2)p+\frac{3m}{2p}\\
 & \leqslant-\left(m+\frac{1}{2}\right)-\frac{3\left(m+\frac{1}{2}\right)}{p}-(m+1)q+2(m+1)+\frac{3m}{2p}   \qquad \text{ by } \eqref{eq-ineq}\\
 & =m+\frac{3}{2}-\frac{3m+1}{2p}-(m+1)q\\
 & <m+\frac{3}{2}-(m+1)q\\
 & <-q-m,
\end{align*}
implying $b_1\notin \Delta D$. Therefore,  $(m+c)\mathbf{v}\notin T-T$ for all integers $c\geqslant 2$.

\noindent \vspace{12pt}

\noindent $\bullet$\;\textbf{Part III}

\noindent In this part we assume that $f(x)=x^{2}+px+q$ with $\Delta=p^{2}-4q\geqslant0$, where $1\leqslant p\leqslant q$ and $q\geqslant4$. Let $D=\{0,1,\dots,q-2,q\}$, then $\triangle D=\{0,\pm1,\dots, \pm(q-1),\pm q\}$. Here we consider the following cases one by one:
\begin{enumerate}
\item[(a)] $q=2p-2$ and $p\geqslant3$
\item[(b)] $q>2p-2$
\item[(c)] $q=p\geqslant4$
\item[(d)] $2p-2>q>p\geqslant3$
\end{enumerate}

\noindent (a)  Notice first that $0=f(A)=A^{2}+pA+qI=(A+I)[A+(p-1)I]+(p-1)I$.
So $A+(p-1)I=-(p-1)(A+I)^{-1}$ and then
\[
A\mathbf{v}+(p-1)\mathbf{v}=-(p-1)(A^{-1}-A^{-2}+A^{-3}-\cdots)\mathbf{v}.
\]
Hence
\begin{equation}\label{eq-iden}
\mathbf{v}=-(p-1)A^{-1}\mathbf{v}+\sum_{i=2}^{\infty}(-1)^{i-1}(p-1)A^{-i}\mathbf{v},
\end{equation}
which implies that $\mathbf{v}\in T-T$. Moreover, by $q=2(p-1)$, we get
$$2\mathbf{v}=-qA^{-1}\mathbf{v}+\sum_{i=2}^{\infty}(-1)^{i-1}qA^{-i}\mathbf{v}\in T-T.$$
Hence $T$ is connected by Lemma \ref{e-connected prop}. (See Figure \ref{fig2}(a))

\vspace{12pt}

\noindent (b)  By Lemma \ref{lem-LiuLuoXie},  then $\tilde{\alpha}=\dfrac{p-1}{q-p+1}$ and $\tilde{\beta}=\dfrac{1}{q-p+1}$.
Also $|\gamma|\leqslant\dfrac{q(p-1)}{q-p+1}$ and $|\delta|\leqslant\dfrac{q}{q-p+1}<2$.

Suppose $2\mathbf{v}=\sum_{i=1}^\infty b_i A^{-i}{\mathbf v}\in T-T$. Then $2A\mathbf{v}-b_{1}\mathbf{v}=\sum_{i=2}^\infty b_i A^{-i}{\mathbf v}\in T-T$. This is impossible as $|\delta|<2$. Hence $2\mathbf{v}\notin T-T$. Similarly we can show that $k\mathbf{v}\notin T-T$ for $k>2$. Hence $T$ is disconnected by Lemma \ref{e-connected prop}.

\vspace{12pt}

\noindent (c)  In this case $|\gamma|\leqslant q(q-1)$ and $|\delta|\leqslant q$. Suppose $k\mathbf{v}=\sum_{i=1}^{\infty}b_{i}A^{-i}\mathbf{v}\in T-T$ where $k\geqslant 2$. By \eqref{neighbor-generating}, then
\begin{equation*}
 -(kq+b_{1})A\mathbf{v}-(kq+b_{2})\mathbf{v}\in T-T.
\end{equation*}
$|kq+b_{1}|\leqslant q$ is only possible when $k=2$. In this case $b_{1}=-q$, so  $-qA\mathbf{v}-(2q+b_{2})\mathbf{v}\in T-T$. By using \eqref{neighbor-generating} again, we have
\begin{equation*}
(q^{2}-2q-b_{2})A\mathbf{v}+(q^{2}-b_{3})\mathbf{v}\in T-T.
\end{equation*}
Since $|q^{2}-2q-b_{2}|\leqslant q$ and $|q^{2}-b_{3}|\leqslant q(q-1)$, we have $q^{2}-q\geqslant b_{2}\geqslant q^{2}-3q=q(q-3)\geqslant q$,
where the last equality holds when $q=4$. Thus if $q=p\geqslant5$, then $b_{2}>q$, impossible. So $2\mathbf{v}\notin T-T$.

If $q=p=4,$ then by $0=f(A)(A-I)=A^{3}+3A^{2}-4I$ we have $A^{2}+2A=2I=2(A+I)^{-1}$.
It follows that
$$v=-2A^{-1}v+2A^{-2}v+2A^{-3}v-2A^{-4}v+2A^{-5}v-2A^{-6}v+\cdots\in T-T$$
and
$$2v=-4A^{-1}v+4A^{-2}v+4A^{-3}v-4A^{-4}v+4A^{-5}v-4A^{-6}v+\cdots\in T-T.$$
Hence $T$ is connected by Lemma \ref{e-connected prop}.

\vspace{12pt}

\noindent (d)  Let $q-p=r\geqslant1$. In this case $|\gamma|\leqslant\dfrac{q(q-r-1)}{r+1}$
and $|\delta|\leqslant\dfrac{q}{r+1}$. Suppose $2\mathbf{v}=\sum_{i=1}^{\infty}b_{i}A^{-i}\mathbf{v}\in T-T$. By \eqref{neighbor-generating}, then
\begin{equation*}
-(2p+b_{1})A\mathbf{v}-(2q+b_{2})\mathbf{v}\in T-T.
\end{equation*}
Since $|2p+b_{1}|\leqslant\dfrac{q}{r+1}$ and $|2q+b_{2}|\leqslant\dfrac{q(q-r-1)}{r+1}$,
we have $-\dfrac{q}{r+1}\leqslant2(q-r)+b_{1}\leqslant\dfrac{q}{r+1}$.
Thus
$$b_{1}\leq\frac{q}{r+1}-2(q-m)=\frac{-(2m+1)q}{m+1}+2m<-q,$$
as $2(r+1)<q$ (ie. $q<2p-2$). That is impossible, hence  $2\mathbf{v}\notin T-T$.

Now suppose $k\mathbf{v}=\sum_{i=1}^{\infty}b_{i}A^{-i}\mathbf{v}\in T-T$ for $k\geqslant3$. As $|kp+b_{1}|\leqslant\dfrac{q}{r+1}$, we have
\begin{align*}
b_{1} & \leqslant\frac{q}{r+1}-kp =\frac{q}{r+1}-k(q-r)\\
 & <\frac{q}{2}-kq+k\left(\frac{q}{2}-1\right)\mbox{\;(as \ensuremath{q>2(r+1)})}\\
 & \leqslant-q-k <-q,
\end{align*}
impossible. So $k\mathbf{v}\notin T-T$ which implies that $T$ is disconnected by Lemma \ref{e-connected prop}. (See Figure \ref{fig2}(b))

\vspace{12pt}

\noindent $\bullet$\;\textbf{Part IV}

\noindent In this part we assume that $f(x)=x^{2}+px+q$ with $\Delta=p^{2}-4q<0$, where $q>p\geqslant0$ and $q\geqslant4$. Let $D=\{0,1,2,\dots,q-2,q\}$, then $\triangle D=\{0,\pm1,\dots, \pm(q-1),\pm q\}$. The proof of this part is divided into the following three cases:
$$(a)\; q=p+1; \quad (b)\; q=p+2;  \quad \text{and}\quad (c)\; q>p+2.$$

\noindent (a)  Notice that $0>p^{2}-4(p+1)=(p-2)^{2}-8$. Now $p=1,2,3,4$. The corresponding $q$'s are $2,3,4,5$. As the cases $(p,q)=(1,2)$
and $(2,3)$ have been solved, we need only study $(p,q)=(3,4)$ and $(4,5)$.

When $(p,q)=(3,4)$, we can deduce from $0=f(A)=A^{2}+3A+4I$ that $A+2I=-2(A+I)^{-1}$. It in turn implies that
$$v=-2A^{-1}-2A^{-2}+2A^{-3}-2A^{-4}+2A^{-5}-2A^{-6}+\cdots\in T-T$$
and
$$2v=-4A^{-1}-4A^{-2}+4A^{-3}-4A^{-4}+4A^{-5}-4A^{-6}+\cdots\in T-T.$$  Hence $T$ is connected.

When $(p,q)=(4,5)$, we can calculate the upper bounds $\tilde{\beta}<0.6$ and $|\delta|<3$ by using the formulas in Section 2. Now suppose $k\mathbf{v}=\sum_{i=1}^{\infty}b_{i}A^{-i}\mathbf{v}\in T-T$ where $k\geqslant 2$. By \eqref{neighbor-generating}, then
\begin{align*}
-(4k+b_{1})A\mathbf{v}-(5k+b_{2})\mathbf{v}\in T-T,
\end{align*}
in which $|\delta|=|4k+b_{1}|<3$, but it is impossible for $k\geqslant 2$ and  $|b_{1}|\leqslant 5$.  Hence  $k\mathbf{v}\notin T-T$, and $T$ is not connected. (See Figure \ref{fig2}(b))

\vspace{12pt}

\noindent (b)  $\Delta=p^{2}-4q<0$ gives $0>p^{2}-4(p+2)=(p-2)^{2}-12$.  Thus $p=2,3,4,5$ and so $q=4,5,6,7$, respectively. We now discuss them in detail.

In the case that $(p,q)=(2,4)$,  we get the upper bounds $\tilde{\beta}<0.43$ and $|\delta|<1.72$. Now suppose $k\mathbf{v}=\sum_{i=1}^{\infty}b_{i}A^{-i}\mathbf{v}\in T-T$. Then $kA\mathbf{v}-b_{1}\mathbf{v}\in T-T$, while $|\delta|< 2$, a contradiction.  Thus $k\mathbf{v}\notin T-T$. Similarly for the case $(p,q)=(3,5)$, we get $\tilde{\beta}<0.38$, and $|\delta|<1.9<2$. Hence $k\mathbf{v}\notin T-T$.

\vspace{12pt}

For the case $(p,q)=(4,6)$, it follows from $0=f(A)=A^{2}+4A+6I$ that $A+3I=-3(A+I)^{-1}$. Then we have
$$v=-3A^{-1}v-3A^{-2}v+3A^{-3}v-3A^{-4}v+3A^{-5}v-\cdots\in T-T$$
and
$$2v=-6A^{-1}v-6A^{-2}v+6A^{-3}v-6A^{-4}v+6A^{-5}v-\cdots\in T-T.$$
So $T$ is connected by Lemma \ref{e-connected prop}. (See Figure \ref{fig2}(a))

\vspace{12pt}

For the case $(p,q)=(5,7)$,  we can get $\tilde{\beta}<0.4$ and $|\delta|<2.8$. Now suppose $k\mathbf{v}=\sum_{i=1}^{\infty}b_{i}A^{-i}\mathbf{v}\in T-T$. Then $kA\mathbf{v}-b_{1}\mathbf{v}\in T-T$. By \eqref{neighbor-generating}, so
\begin{align*}
 -(5k+b_{1})A\mathbf{v}-(7k+b_{2})\mathbf{v}\in T-T.
\end{align*}
Since $|5k+b_{1}|<2.8$, we have $-12.8<b_{1}<-7.2$, then $b_1\notin \triangle D$. Thus $k\mathbf{v}\notin T-T$.

\vspace{12pt}

\noindent (c)  Notice that
\begin{gather}\label{eq-estimate}
|\delta|\leqslant\frac{q}{q-p-1}\mbox{\;\;and\;\;}|\gamma|\leqslant\begin{cases}
\dfrac{q(p-1)}{q-p-1} & \mbox{\;if \ensuremath{p^{2}-q\geqslant0}}\\
\dfrac{qp-2p^{2}+q}{q-p-1} & \mbox{\;if \ensuremath{p^{2}-q<0}.}
\end{cases}
\end{gather}
Now suppose  $k\mathbf{v}=\sum_{i=1}^{\infty}b_{i}A^{-i}\mathbf{v}\in T-T$. As before, by \eqref{neighbor-generating}, then
\begin{equation}\label{eq-2v}
  -(kp+b_{1})A\mathbf{v}-(kq+b_{2})\mathbf{v}\in T-T.
\end{equation}
Since $q>p+2$, from \eqref{eq-estimate}, we obtain
$$
|kp+b_{1}|\leqslant \frac{q}{q-p-1}\leqslant\frac{q}{2}.
$$
Hence $k\mathbf{v}\notin T-T$ provided that $q< \dfrac{4p}{3}$. For the case $q\geqslant\dfrac{4p}{3}$. Here we consider two subcases:
(i) $p^{2}\geqslant q$ and  (ii) $p^{2}< q$.

\vspace{12pt}

(i)  From \eqref{eq-estimate}, it follows that $|kq+b_{2}|\leqslant\dfrac{q(p-1)}{q-p-1}$. If $q>2p$ then $b_{2}\leqslant\dfrac{q(p-1)}{q-p-1}-2q<-q$. Thus $k\mathbf{v}\notin T-T$ and $T$ is disconnected. If $2p\geqslant q\geqslant\dfrac{4p}{3}$, we can find all possible $(p,q)$'s  as follows: $$(2,4),(3,4),(3,6),(4,6),(4,7),(4,8),(5,7),(5,8),$$
$$(5,9),(5,10),(6,10),(6,11),(6,12),(7,13),(7,14).$$
Except the solved cases, we only need to study the following ones:
$$(3,6),(4,7),(4,8),(5,9),(5,10),(6,11),(6,12),(7,13),(7,14).$$

By calculating the upper bounds of $\tilde{\beta}$'s one by one carefully, we can take $\tilde{\beta}< 0.3$ for $(3,6)$; $\tilde{\beta}< 0.26$ for $(4,7)$; $\tilde{\beta}< 0.21$ for $(4,8), (5,9)$; $\tilde{\beta}< 0.17$ for $(5,10), (6,11)$; $\tilde{\beta}< 0.15$ for $(6,12), (7,13)$; $\tilde{\beta}< 0.13$ for $(7,14)$. While in each case $|\delta|<2$ always holds. Thus $k{\mathbf v}\notin T-T$.

\vspace{12pt}

(ii)  From \eqref{eq-estimate},  we have $|kq+b_{2}|\leqslant\dfrac{qp-2p^{2}+q}{q-p-1}$. So $b_{2}\leqslant\dfrac{qp-2p^{2}+q}{q-p-1}-kq\leqslant\dfrac{qp-2p^{2}+q}{q-p-1}-2q$. It can be shown that
\[
\frac{qp-2p^{2}+q}{q-p-1}-2q<-q \quad \Leftrightarrow \quad 2+2p\left(1-\frac{p}{q}\right)<q.
\]

Now for any $p\geqslant3$ we have $p^{2}-\left[2+2p\left(1-\dfrac{p}{q}\right)\right]=(p-1)^{2}-3+\dfrac{2p^{2}}{q}>0$.
So $2+2p\left(1-\dfrac{p}{q}\right)<p^{2}<q$. Thus $k{\mathbf v}\notin T-T$ and $T$ is disconnected. We then study what will happen when $p=1$ or $2$.

\vspace{12pt}

In the case that $p=1$,  $p^{2}-q=1-q<0$ and $q\geqslant\dfrac{4p}{3}=\dfrac{4}{3}$. Then $$|\gamma|\leqslant\dfrac{2q-2}{q-2}=2+\dfrac{2}{q-2}\leqslant3$$ for any $q\geqslant4$ whereas the equality holds when $q=4$. Also
$$|\delta|\leqslant\dfrac{q}{q-2}=1+\dfrac{2}{q-2}\leqslant2$$
whereas the equality holds when $q=4$. Now suppose $k\mathbf{v}\in T-T$. By \eqref{eq-2v}, it follows from $|kq+b_{2}|\leqslant3$ that
$b_{2}\leqslant 3-kq \leqslant3-2q<-q$, then $b_2\notin \triangle D$. Hence $k\mathbf{v}\notin T-T$ and  $T$ is disconnected.

\vspace{12pt}

In the case that $p=2$,  $p^{2}-q=4-q<0$. So $q\geqslant 5$. Then $|\gamma|\leqslant\dfrac{3q-8}{q-3}<4$
and $|\delta|\leqslant\dfrac{q}{q-3}=1+\dfrac{1}{q-3}$. Now suppose
$k\mathbf{v}\in T-T$. By \eqref{eq-2v}, then
$$
 -(2k+b_{1})A\mathbf{v}-(kq+b_{2})\mathbf{v}\in T-T.
$$
From $|kq+b_{2}|<4$, it follows that $b_{2}< 4-kq\leqslant 4-2q<-q$. Then $b_2\notin \triangle D$. Hence $k\mathbf{v}\notin T-T$ and  $T$ is disconnected.

\begin{figure}[h]
  \centering
   \subfigure[$(p,q)=(4,6)$]{
  \includegraphics[width=5cm]{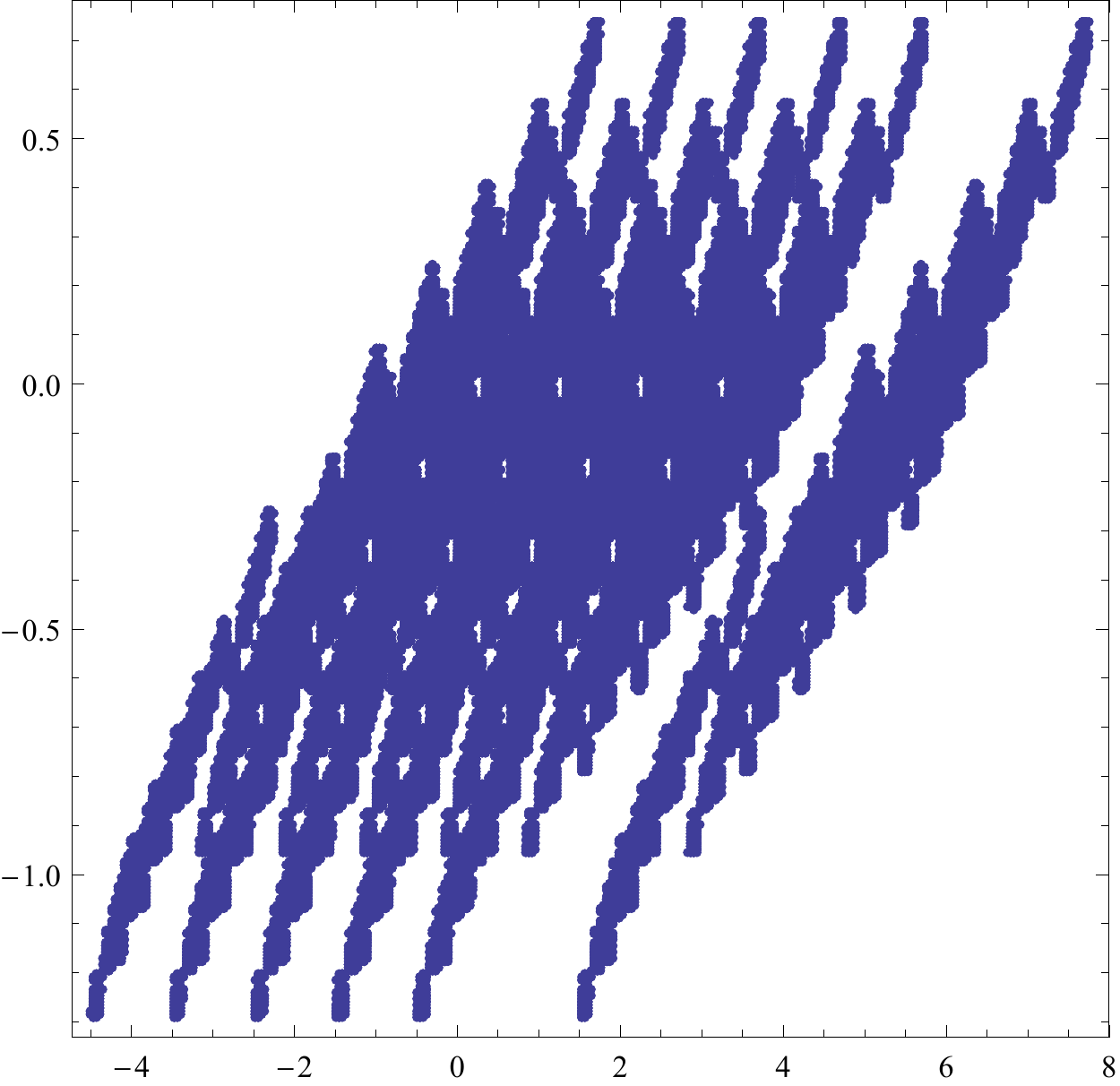}
 }
 \qquad
 \subfigure[$(p,q)=(4,5)$]{
  \includegraphics[width=5cm]{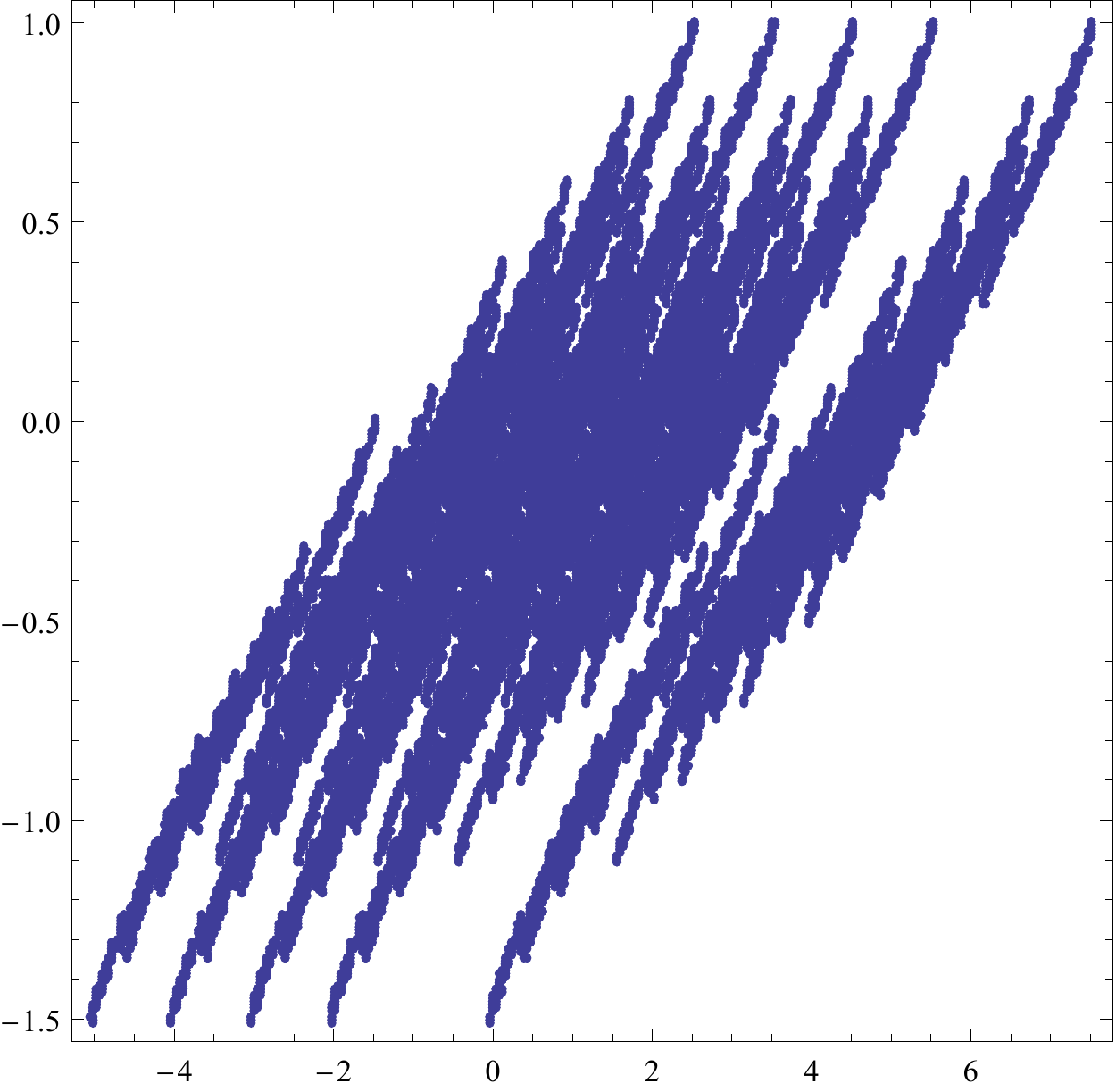}
 }
 \caption{(a) is connected while (b) is not connected.}\label{fig2}
\end{figure}

\vspace{12pt}

\noindent $\bullet$\;\textbf{Part V}

\noindent In this part we assume that $f(x)=x^{2}+px-q$, where $q>p\geqslant1$ and $q\geqslant4$. Let $D=\{0,1,2,\dots,q-2,q\}$.
As $A$ is expanding, $q\geqslant p+2$ by \eqref{eq.expanding}. Since $\tilde{\alpha}=\dfrac{p+1}{q-p-1}$ and $\tilde{\beta}=\dfrac{1}{q-p-1}$, we have
$$
|\gamma|\leqslant\frac{q(p+1)}{q-p-1} \quad\text{and}\quad |\delta|\leqslant\frac{q}{q-p-1}.
$$
Under the assumption that {$p\geqslant 1$},  we divide the proof into the following cases:
$$(a)\  q=2p+2;\quad (b) \  q>2p+2; \quad\text{and}\quad  (c) \  2p+2>q\geqslant p+2.$$

\noindent (a)  Notice that $0=f(A)=A^{2}+pA-(2p+2)I$. So
\begin{align*}
 & (p+1)I=[A+(p+1)I](A-I)\\
\Rightarrow\quad & A+(p+1)I=(p+1)(A-I)^{-1}=(p+1)\sum_{i=1}^{\infty}A^{-i}\\
\Rightarrow\quad & \mathbf{v}=-(p+1)A^{-1}\mathbf{v}+\sum_{i=1}^{\infty}(p+1)A^{-i-1}\mathbf{v}\in T-T.
\end{align*}
Moreover, $2\mathbf{v}=-qA^{-1}\mathbf{v}+\sum_{i=2}^{\infty}qA^{-i}\mathbf{v}\in T-T$. Thus $T$ is connected.

\vspace{12pt}

\noindent (b)  Notice that $q-p-1>p+1$. We get
\begin{equation}\label{eq-gamma}
|\gamma|<q
\end{equation}
Now suppose $k\mathbf{v}=\sum_{i=1}^{\infty}b_{i}A^{-i}\mathbf{v}\in T-T$ where $k\geqslant 2$. By \eqref{neighbor-generating}, then
\begin{align*}
-(kp+b_{1})A\mathbf{v}+(kq-b_{2})\mathbf{v}\in T-T.
\end{align*}
By \eqref{eq-gamma}, we know  $|kq-b_{2}|<q$. Then $|b_{2}|>(k-1)q\geqslant q$, contradicting $b_2\in \triangle D$. Thus $k\mathbf{v}\notin T-T$.

\vspace{12pt}

\noindent (c)  In this case $|\gamma|\leqslant q(p+1)$ and $|\delta|\leqslant q$. Suppose $k\mathbf{v}=\sum_{i=1}^{\infty}b_{i}A^{-i}\mathbf{v}\in T-T$. Then $-(kp+b_{1})A\mathbf{v}+(kq-b_{2})\mathbf{v}\in T-T$, which then implies that
\[
[p(kp+b_{1})+(kq-b_{2})]A\mathbf{v}-[q(kp+b_{1})+b_{3}]\mathbf{v}\in T-T.
\]
Since $|p(kp+b_{1})+(kq-b_{2})|\leqslant q$, we have $p(2p+b_{1})+q\leqslant p(kp+b_{1})+q \leqslant b_{2}$. When $q<2p$, as $-2p<-q\leqslant b_{1}$,  then $q<b_{2}$, contradicting $b_2\in \triangle D$. Thus $k\mathbf{v}\notin T-T$. When $2p\leqslant q<2p+2$, which means (i) $q=2p$ or (ii) $q=2p+1$.

\vspace{12pt}

(i)  Notice that $|p(kp+b_{1})+(kq-b_{2})|<q$. So $\left|\dfrac{q}{2}(\frac{kq}{2}+b_{1})+(kq-b_{2})\right|<q$, implying that $\dfrac{q}{2}(q+b_{1})+q\leqslant \frac{q}{2}(\frac{kq}{2}+b_1)+(k-1)q<b_{2}$. Since $b_{1}\geqslant-q$, we have $q<b_{2}$, and $b_2\notin \triangle D$. Hence $k\mathbf{v}\notin T-T$.

\vspace{12pt}

(ii) Notice that $|p(kp+b_{1})+(kq-b_{2})|\leqslant\dfrac{q}{q-p-1}=2+\dfrac{2}{q-1}$. Seeing that $\left|\dfrac{2}{q-1}\right|<1$, hence
$|p(kp+b_{1})+(kq-b_{2})|\leqslant2$. Now
\begin{align*}
 & p(2p+b_{1})+2q-b_{2}\leqslant p(kp+b_{1})+kq-b_{2}\leqslant 2\\
\Rightarrow\quad & \frac{q-1}{2}(q-1+b_{1})+2q-b_{2}\leqslant2\\
\Rightarrow\quad & \frac{(q-1)^{2}}{2}+\frac{q-1}{2}b_{1}+2q-2\leqslant b_{2}\\
\Rightarrow\quad & \frac{(q-1)^{2}}{2}-\frac{q(q-1)}{2}+2q-2\leqslant b_{2}  \quad (as\,-q\leq b_{1})\\
\Rightarrow\quad & \frac{3q}{2}-\frac{3}{2}\leqslant b_{2}\\
\Rightarrow\quad & q<b_{2}\quad \mbox{\;(as \ensuremath{q>3})},
\end{align*}
and $b_2\notin \triangle D$. Therefore, $k\mathbf{v}\notin T-T$ and $T$ is not connected.

\end{document}